\documentclass{amsart}
\usepackage{amsfonts}
\usepackage{amsthm}
\usepackage{amsmath}
\usepackage{amsfonts}
\usepackage{latexsym}
\usepackage{amssymb}
\usepackage[latin1]{inputenc}
\usepackage{verbatim}

\newcommand{\U}{\mathcal{U}}

\newcommand{\n}{\mathcal{N}}

\newcommand{\um}{\U_{\cM}}

\def\um{\mathcal{U}_\m}

\def\a{\mathcal A}
\def\b{\mathcal B}

\newcommand{\CC}{\mathbb{C}}
\newcommand{\RR}{\mathbb{R}}
\newcommand{\NN}{\mathbb{N}}

\newtheorem{fed}{Definition}[section]
\newtheorem{teo}[fed]{Theorem}
\newtheorem*{teo*}{Theorem}
\newtheorem{lem}[fed]{Lemma}
\newtheorem{cor}[fed]{Corollary}
\newtheorem{pro}[fed]{Proposition}
\theoremstyle{definition}
\newtheorem{rem}[fed]{Remark}
\newtheorem{rems}[fed]{Remarks}

\newcommand{\topologia}{\,\mbox{\tiny sot}}

\def\m{\mathcal{M}}
\def\msa{\m^{\text{sa}}}

\def\cD{\mathcal{D}}

\def\cP{\mathcal{P}}

\newcommand{\ds}{\displaystyle}

\begin{document}

\title{Towards the Carpenter's Theorem}
\author{Martín Argerami}
\address{Department of Mathematics, University of Regina, Regina SK, Canada}
\email{argerami@math.uregina.ca}
\thanks{M. Argerami supported in part by the Natural Sciences and Engineering Research
Council of Canada}
\author{Pedro Massey}
\address{Departamento de Matem\'atica, Universidad Nacional de La Plata and
Instituto Argentino de Matem\'atica-conicet, Argentina}
\email{massey@mate.unlp.edu.ar}
\thanks{P. Massey supported in part by CONICET of Argentina and a PIMS Postdoctoral Fellowship}



\begin{abstract}
Let $\m$ be a II$_1$ factor with trace $\tau$,
$\a\subseteq \m$  a masa and  $E_\a$   the unique
conditional expectation onto $\a$. Under some technical
assumptions on the inclusion $\a\subseteq \m$, which hold true for
any semiregular masa of a separable factor, we
show that for elements $a$ in certain dense families of the positive part of the unit ball of $\a$,
it is possible to find a projection $p\in\m$ such that $E_\a(p)=a$.
This shows a new family of instances of a
conjecture by Kadison, the so-called ``carpenter's theorem''.
\end{abstract}
\maketitle

\section{Introduction}

As it is well-known, the Pythagorean Theorem (PT) states that the square
of the norm of the sum of two orthogonal vectors is equal to the
sum of the square of the norms of each vector. A converse of the
theorem would be the statement that if such equality occurs, then
the two vectors were orthogonal to begin with. Such a result
allows a carpenter to check his right-angles by just measuring
length, so that's why PT's converse is called the
``carpenter's theorem'' (CT) by Kadison. In his work \cite{Kad0, Kad1}, he considers extensions of PT
and its corresponding converses CT to infinite dimension, getting to the unexpected and
striking Theorem 15 in \cite{Kad1} (extended by Arveson in \cite{arv-diag}).
These generalizations of PT and CT are carried in \cite{Kad0} to
the realm of II$_1$ factors, where the PT basically becomes
tautological, and the CT becomes the following:

\smallskip

\noindent{\bf
Conjecture }[Kadison's carpenter's theorem]\label{conj Kad} Let
$\a$ be a masa of the II$_1$ factor $\m$ and let $a\in \a^+_1$.
Then there exits a projection $p\in \cP(\m)$ such that
$E_\a(p)=a$, where $E_\a$ denotes the trace preserving conditional
expectation onto $\a$.

\medskip

In the finite dimensional case, the CT is a particular case of the
well-known Schur-Horn theorem. Whether the Schur-Horn theorem
extends or not to II$_1$ factors is unknown at the moment (see
\cite{argmas, arvkad}). In this paper we focus on the CT in II$_1$
factors. Assuming some restrictions on the factor and the masa,
which hold true for semiregular masas in separable II$_1$ factors,
we show that the statement holds for various dense families. It is
worth mentioning here that the statement of the CT (and also of
Schur-Horn) is only meaningful in the case of masas, for this
would imply the result for any other abelian subalgebra, and also
because both statements are likely to fail when the subalgebra
considered is not abelian: indeed, CT does not hold for
non-abelian subalgebras of $M_n(\CC)$, and so neither does
Schur-Horn.

Although our results fail to settle the CT conjecture in full
generality, our methods lead us to consider a possible strategy for
obtaining the CT under the conditions we consider for the inclusion
$\mathcal A\subseteq \mathcal M$, as explained at the end of the
paper. It is worth noting that these technical conditions hold true
for inclusions $\mathcal A\subseteq \mathcal M$ where $\mathcal A$
is semiregular.

\section{Preliminaries}\label{dos}

Throughout the paper $\m$ denotes a II$_1$ factor with normalized
faithful normal trace $\tau$. We denote by $\msa$, $\m^+$, $\um$,
the sets of selfadjoint, positive, and unitary elements of $\m$. By
$\cP(\m)$ we mean the set of projections of $\m$. Given $a\in \msa$
we denote its spectral measure by $p^a$; thus, $p^a(\Delta)$ is the
spectral projection associated with a Borel set $\Delta\subset \RR$.
The characteristic function of the set $\Delta$ is denoted by
$\chi^{\phantom{A}}_\Delta$ and its Lebesgue measure by $m(\Delta)$.
The unitary orbit of $a \in\m^{sa}$ is the set $\U_\m(a)=\{ uau^* :
u \in \U_\m\}.$

\smallskip

In \cite{Kad0}, Kadison conjectured that if $\a\subseteq \m$ is a
masa and $a\in \a^+_1$ i.e., $a\in \a^+$ and $0\leq a\leq 1$, then
there exits a projection $p\in \cP(\m)$ such that $E_\a(p)=a$. This
conjecture is equivalent to the following assertion: for
$p\in\cP(\m)$, $a\in\a$,
\begin{equation}\label{equiv Kad}
0\le a\le1, \tau(a)=\tau(p)\ \Leftrightarrow\ a\in E_\a(\um(p)).
\end{equation}

Using \eqref{equiv Kad} it can be shown that Kadison's conjecture is
a particular case of a more general conjecture (a Schur-Horn theorem
in II$_1$ factors), that was stated as an open problem by Arveson
and Kadison in \cite{arvkad}. In \cite{argmas} we proved a weaker
version of Arveson-Kadison's conjecture, that restricted to the
situation in \eqref{equiv Kad} is
\begin{teo}\label{noso}
Let $\a\subseteq \m$, $a\in\a$, $p\in \cP(\m)$.
Then $$ 0\leq a\leq1, \ \ \tau(a)=\tau(p) \ \ \Leftrightarrow \ \ a\in \overline{E_\a(\U_\m(p))}^{\topologia}$$
\end{teo}
Note that in \eqref{equiv Kad} the unitary orbit of the projection
is already strongly closed (and so norm-closed, too), but the
statement in Theorem \ref{noso} is weaker because it is not clear
whether the set on the right-hand side of \eqref{equiv Kad} is
already closed in the strong operator topology (a fact that is
actually equivalent to Kadison's conjecture by Theorem \ref{noso}).

\medskip

\noindent{\bf Matrix Units.} Given a masa $\a$ in $\m$, we denote by $\n_\a$  the normalizer of $\a$ in $\m$, i.e.
 the subgroup of $\um$ given by
 \[
 \n_\a=\{u\in\um: \ u^*\a u=\a\}.
\]
The masa $\a$ is said to be {\bf semiregular} if $(\n_\a)''$ is a factor, and {\bf regular} (or Cartan) if
$(\n_\a)''=\m$. Popa shows in \cite[Proposition 3.6]{popa-kad} that any semiregular masa in a
separable type II factor is Cartan in a hyperfinite subfactor.
His result implies the following:
\begin{pro}\label{iso parc} If $\a\subset \m$ is a semiregular masa in the
separable II$_1$ factor $\m$, then for every $k\in \NN$ there exists
$\{u_i^k\}_{i=1}^{2^k}\subset \n_\a$ and
$\{p^k_i\}_{i=1}^{2^k}\subset\cP(\a)$ such that $\{v^k_{ij}\}_{ij}$,
where $v^k_{ij}=u_i^kp_1^k(u_j^k)^*$,  is a $2^k$-system of matrix
units with $v^k_{jj}=p_j^k\in\cP(\a)$ for $j=1,\ldots,2^k$ and such
that
\begin{equation}\label{tensor}
v_{2i-1,2j-1}^{k+1}+v_{2i,2j}^{k+1}= v_{ij}^k \ , \ \ \ 1\leq i,\ j
\leq 2^k, \end{equation} and such that the family $\{p^k_j\}$
generates all of $\a$.
\end{pro}

Matrix units can always be constructed in a II$_1$ factor, but the result in Proposition \ref{iso parc}
allows one to make ``coherent embeddings'',
in a sense made precise in Corollary \ref{repre}.

We denote by $\cD(n)$ the diagonal subalgebra of  $M_n(\CC)$
 and by $E_{\cD(n)}:M_n(\CC)\rightarrow \cD(n)$ the diagonal compression.
We also consider $\phi_k:M_{2^ k}(\CC)\rightarrow
M_{2^{k+1}}(\CC)$ to be the unital *-monomorphism $\phi_k(A)= I_2\otimes A$.
Denote by $\{e^k_{ij}\}$ the canonical
matrix units in $M_{2^k}(\CC)$.

\begin{cor}\label{repre}
Let $\{p^k_j\}$, $\{v^k_{ij}\}$ be
as in Proposition \ref{iso parc}. Define a
family of $*$-monomorphisms $\pi_k:M_{2^k}(\CC)\rightarrow \m$
in the following way: for $a=(a_{ij})\in M_{2^k}(\CC)$, let
\[
\pi_k(a)=\sum_{i,j}\, a_{ij} v^k_{ij}.
\]
Then $\pi_k(e^k_{ii})=p_i^k$ for $i=1,\ldots,2^k$, and
$\pi_k=\pi_{k+1}\circ\phi_k $, $\pi_k\circ E_{\mathcal
D(2^k)}=E_\a\circ\pi_k$, $k\in\NN$.
\end{cor}

For every
$k\in \NN$ let $\{I_i^k\}_{i=1}^{2^k}$ denote the dyadic partition
of $[0,1]$ given by $I^k_i=[(i-1)2^{-k},i\,2^{-k})$.

\begin{rem} \label{el x}
To each family $\{\,\{p_i^k\}_{i=1}^{2^k}:\ k\in\NN\}\subseteq \a$
as in Proposition \ref{iso parc} we associate an operator $x$ in the
following way. It is easy to see that the sequence of discrete
positive operators $x_k=\sum_{i=1}^{2^k} \frac{i}{2^k} \ p_i^k\in
\a^+$ is non-increasing and bounded. Let $x=\lim_{\rm SOT}x_k \in
\a^+$. Then, for every $k\in \NN$ and $0\leq i\leq 2^k$,
$p^x(I_i^k)=p_i^k$. In particular, $\tau\circ p^x$ is the Lebesgue
measure restricted to $[0,1]$.
   We say that $x$ is {\bf the associated operator} to the family
   $\{p_i^k\}$. Notice that the von Neumann sub-algebra generated by $x$
   coincides with $\a$, since the projections $p_j^k$ are Borel
   functional calculus of $x\in \a$.
\end{rem}

\section{Main results}\label{cuatro}

Two subalgebras $\a,\b\subset\m$ are said to be {\bf orthogonal} \cite{popa-ortog} in $\m$ if
$E_\a(\b)\subset\CC\,I$.

\begin{fed}
We say that a masa $\a\subset\m$ is {\bf totally complementable} if for every
projection $p\in\a$, the masa $p\a$ in $p\m p$ admits a diffuse
orthogonal subalgebra.
\end{fed}

\begin{teo}[Carpenter's theorem for discrete  operators]\label{teo tipo disc}
If $\a$ is a totally complementable masa in the II$_1$ factor $\m$, then for every discrete $a\in(\a)^+_1$
there exists a projection $p\in\m$ such that $E_\a(p)=a$.
\end{teo}
\begin{proof}
Assume $\b\subset\m$ is a subalgebra orthogonal to $\a$.
For any $\alpha\in[0,1]$, there exists a projection $q\in\b$ with
$\tau(q)=\alpha$. Since $\a$ and $\b$ are orthogonal,
$E_\a(q)=\tau(E_\a(q))=\tau(q)=\alpha$.

Now let $p\in\a$ be a projection; then $p\a$ is a masa in $p\m p$,
so it admits an orthogonal subalgebra $\b_p$. By the first paragraph, there exists a
projection $q\in \b_p\subset p\m p$ with $E_{p\a}(q)=\alpha\,p$. Since $q\in  p\m
p$, in particular $q=pq$. So
\[E_\a(q)=E_\a(pq)=p\,E_\a(q)=E_{p\a}(q)=\alpha\,p.\]

Now let $a=\sum_k \alpha_k\,p_k$ where $\{p_k\}_{k\in\NN}$ is a
sequence of mutually orthogonal projections in $\a$ and
$\{\alpha_k\}_{k\in\NN}$ is a
 sequence of numbers. Since $0\leq a\leq 1$, we have $0\leq \alpha_k\leq 1$.
For each $k\in \NN$ apply the first part of the proof to get a
projection $q_k\in\m^+$ such that $E_\a(q_k)=\alpha_k \,p_k$,
$q_k\leq p_k$. Thus, the operator $q=\sum_kq_k\in \m$ is a
projection such that $E_\a(q)=\sum_k \alpha_k\,p_k$.
\end{proof}

\begin{rems}
\ \smallskip
\begin{enumerate}
\item The conditions in Theorem \ref{teo tipo disc} are satisfied
by a
Cartan masa of the hyperfinite II$_1$ factor, and so by any semiregular
masa in a separable II$_1$ factor, since it is Cartan in an intermediate hyperfinite subfactor
\cite[Proposition 3.6]{popa-kad}.
\item Because in general
there is no clear ``coherent'' way of constructing the
projections $q_k$ in the previous proof, we would not expect such
argument to be useful to prove the general case of the Carpenter's
theorem.
\item Under the conditions of Theorem \ref{teo tipo
disc}, it follows in particular that there exists a projection
$p\in \a$ such that $$E_\a(p)=\frac{1}{\sqrt{2}} \, I.$$
Remarkably, it seems hard to prove even this particular case of Kadison's
conjecture in the general case of an arbitrary II$_1$ factor and a
masa $\a\subseteq \m$.
\end{enumerate}
\end{rems}

\bigskip

In the remainder of the paper, given a semiregular masa $\a$ of the
separable II$_1$ factor $\m$, we
 will prove the Carpenter's Theorem for some
 non-discrete operators, namely piece-wise linear functional calculus of
$x$, the associated operator of a family of projections  considered
in Remark \ref{el x}.

We begin by defining the following sequence of unitary matrices
$(W_n)_n$:
\begin{equation*}
W_1=\begin{pmatrix} 1 & 0 & 0 & 0 \\
                    0 & \frac{1}{\sqrt{2}}  & -\frac{1}{\sqrt{2}} & 0\\
                    0 & \frac{1}{\sqrt{2}} & \frac{1}{\sqrt{2}} & 0 \\
                    0 & 0 & 0 & 1 \end{pmatrix}, \ \ \ \ \
W_{n+1}=W_n\otimes I_2
=\begin{pmatrix}
W_n&0\\0&W_n
\end{pmatrix}=\bigoplus_{j=1}^{2^{n}}W_1
\end{equation*}

\begin{lem}\label{iteracion}
Let $A\in M_{2^k}(\CC)$. Put
$A(1)=A$, $A(n+1)= W_{k+n-1} (I_2\otimes A(n))
W_{k+n-1}^*$. Then there exists  $\lambda<1 $, independent of $A$, $k$ and $n$ such that
\begin{equation*}
\frac{1}{2} \|A(n+1)-I_2\otimes A(n)\|_2^2\leq \, \lambda\,
\|A(n)-I_2\otimes A(n-1)\|_2^2\end{equation*}
\end{lem}
\begin{proof}
Let $k\geq 1$ and  $n\geq 2$. We can consider $A(n-1)$ as a block matrix with $2\times2$ blocks, i.e.
$A(n-1)=(A_{ij})_{ij}$ where $A_{ij}\in M_2(\CC)$ for $1\leq
i,\,j\leq 2^{(k+n-3)}$. It is easy to verify that
\[I_2\otimes A(n-1)=(I_2\otimes
A_{ij})_{ij}\ \text{ and } \ A(n)=(W_1(I_2\otimes
A_{ij})W_1^*)_{ij}=(A_{ij}(2))_{ij}.\]
 So in particular we have that
\begin{equation}\label{anmenos1} \|A(n)-I_2\otimes A(n-1)\|_2 ^2
=\sum_{i,\,j=1}^{\ \ 2^{(k+n-3)}}\| A_{ij}(2)-I_2\otimes A_{ij}\|_2^2
\end{equation}
 Similarly we see that $ A(n+1)=(A_{ij}(3))_{ij}\, , \ \ \text{ for } 1\leq i,\,j\leq 2^{k+n-3}$ and
 \begin{equation}\label{anmas1} \|A(n+1)-I_2\otimes A(n)\|_2 ^2
 =\sum_{i,\,j=1}^{\ \ 2^{(k+n-3)}}\| A_{ij}(3)-I_2\otimes A_{ij}(2)\|_2^2.
 \end{equation}
So, from \eqref{anmenos1} and \eqref{anmas1} we see that it is enough to prove that
there exists $0<\lambda<1$ (independent of $A$, $k$ and $n$) such that for every $1\leq i,\, j\leq 2^{k+n-3}$,
\begin{equation*}
\frac{1}{2} \|A_{ij}(3)-I_2\otimes A_{ij}(2)\|_2^2 \leq \, \lambda\,
\|A_{ij}(2)-I_2\otimes A_{ij}\|_2^2.
\end{equation*} We show that such inequality holds for any $2\times 2$ matrix $B=(b_{ij})_{ij}\in M_2(\CC)$.
By  straightforward computations,
\begin{equation*} B(2)=W_1\,(I_2\otimes\,B)\,W_1^*=
\begin{pmatrix} b_{11} &
\frac{-b_{12}}{\sqrt{2}} &  \frac{b_{12}}{\sqrt{2}} & 0 \\ &\\
\frac{-b_{21}}{\sqrt{2}} & \frac{b_{11}+b_{22}}{2} &
\frac{b_{11}-b_{22}}{2} & \frac{b_{12}}{\sqrt{2}} \\ &\\
\frac{b_{21}}{\sqrt{2}} & \frac{b_{11}-b_{22}}{2} &
\frac{b_{11}+b_{22}}{2} & \frac{b_{12}}{\sqrt{2}} \\  &\\
0 &
\frac{b_{21}}{\sqrt{2}} &  \frac{b_{21}}{\sqrt{2}} & b_{22}
\end{pmatrix}
\end{equation*}
and so
\begin{equation}\label{eq1 norm} \| B(2)- I_2\otimes
B\|_2^2=
(4-2\sqrt{2})(|b_{12}|^2+|b_{21}|^2)+|b_{11}-b_{22}|^2.\end{equation}

\medskip

Thus, if we consider $B(2)=(B_{ij})_{ij}$ as a $2\times2$ block matrix, where $B_{ij}\in
M_2(\CC)$, we can use the previous calculation with each of these four
matrices and get
\begin{equation}\label{eq2 norm}
\frac{1}{2}\|B(3)-I_2\otimes
B(2)\|_2^2=\frac{1}{2}((4-2\sqrt{2})(|b_{12}|^2+|b_{21}|^2)+(\frac{5}{2}-\sqrt{2})\,|b_{11}-b_{22}|^2
)
\end{equation}
Writing $\frac{5}{2}-\sqrt{2}=1+(\frac{3}{2}-\sqrt{2})$ and
using \eqref{eq1 norm} and \eqref{eq2 norm} we get that
$$\frac{1}{2} \frac{\|B(3)-I_2\otimes B(2)\|_2^2 }{\| W_1
(I_2\otimes B)W_1^*- I_2\times B\|_2^2}\leq \frac{1}{2}(1 +
\frac{3}{2}- \sqrt{2})< 1.\qedhere $$
\end{proof}

In what follows we denote by $\{f^k_i\}_{i=1}^{2^k}$ the rank-one projections associated with the elements of the
canonical basis of $\CC^{2^k}$ that is $f_i^k=e^k_{ii}$.

\begin{lem}\label{lema diagonal}
Let $n\in\NN$ and $A\in M_{2^k}(\CC)$. Then, with the notations of Lemma \ref{iteracion}:

\begin{enumerate}
\item  $ E_{\cD(2^{k+n})}(A(n+1))= E_{\cD(2^{k+n})}(  W_{k+n-1}\,(I_2\otimes E_{\cD(2^{k+n-1})}(A(n)))\,W_{k+n-1}^* ) $
\item\label{conjugado diagonal}If $A$ is diagonal and $B=W^{\phantom{*}}_{k-1}\,A\,W_{k-1}^*$, then
\[
B_{ii}=\left\{\begin{array}{ll}A_{ii}&\mbox{ if }i=4h\mbox{ or }i=4h-3\\
\frac12\,(A_{4h-1,4h-1}+A_{4h-2,4h-2})&\mbox{ if }i=4h-1\mbox{ or }i=4h-2\end{array}\right.
\]
\item\label{diagonal} If
$E_{\cD(2^k)}(A)=\sum_{\ell=1}^{2^k}\,d_\ell\,f^k_{\ell}$, then
\[
\begin{array}{rcl}
E_{\cD(2^{k+n-1})}(A(n))&=&\ds\sum_{\ell=1}^{2^{k-1}}\sum_{h=1}^{2^{n-1}}\,
\gamma^{n}_{\ell,h-1}\, f^{k+n-1}_{2^n(\ell-1)+2h-1}+\gamma^{n}_{\ell,h}\,f^{k+n-1}_{2^n(\ell-1)+2h}
\end{array}
\]

where
\[
\gamma^{n}_{\ell,h}=d_{2\ell-1}+\frac{h}{2^{n-1}}\,(d_{2\ell}-d_{2\ell-1}).
\]
\end{enumerate}
\end{lem}

\begin{proof}
To prove (i) let $k,\,n\geq 1$ and consider the block representations  $A(n)=(A_{ij})_{i,j=1}^{2^{k+n-2}}$ where $A_{ij}\in M_2(\CC)$.
Then $I_2\otimes A(n)=(I_2\otimes A_{ij})_{ij=1}^{2^{k+n-2}}$ and
\[
A(n+1)=W_{k+n-1}(I_2\otimes A(n))W_{k+n-1}^*=(W_1 \,(I_2\otimes A_{ij})\,W_1^*)_{ij=1} ^{2^{k+n-2}}
\]
with respect to the previous block representation. Hence, to study the diagonal of $A(n+1)$ we can restrict
our attention to the diagonal blocks $W_1\,(I_2\otimes A_{ii})\,W_1^*\in M_4(\CC)$, for $i=1,\ldots,2^{k+n-2}$.
Straightforward computations show that $$E_{\cD(4)}( W_1\,(I_2\otimes A_{ii})\,W_1^*)=
 E_{\cD(4)}( W_1\,E_{\cD(4)}(I_2\otimes A_{ii})\,W_1^*)$$ from which (i) follows, after noting that
 $E_{\cD(4)}(I_2\otimes B)=I_2\otimes E_{\cD(2)}(B)$ for any $B\in M_2(\CC)$.

The proof of \eqref{conjugado diagonal} is straightforward.

We prove \eqref{diagonal} by induction. The case $n=1$ follows from the
definitions and hence we omit it. Now, assume that \eqref{diagonal} holds for $A(n)$. Then
\[
\begin{array}{rcl}
I_2\otimes E_{\cD(2^{k+n-1})}(A(n))
&=&\ds\sum_{\ell=1}^{2^{k-1}}\sum_{h=1}^{2^{n-1}}\,
\gamma^{n}_{\ell,h-1}\, I_2\otimes f^{k+n-1}_{2^n(\ell-1)+2h-1} \\ \ \\
&&+\gamma^{n}_{\ell,h}\,I_2\otimes f^{k+n-1}_{2^n(\ell-1)+2h}\\ \ \\
&=&\ds\sum_{\ell=1}^{2^{k-1}}\sum_{h=1}^{2^{n-1}}\,
\gamma^{n}_{\ell,h-1}\, (f^{k+n}_{(\ell-1)2^{n+1}+4h-3}+f^{k+n}_{(\ell-1)2^{n+1}+4h-2}) \\ \ \\
&&+\gamma^{n}_{\ell,h}\,(f^{k+n}_{(\ell-1)2^{n+1}+4h-1}+f^{k+n}_{(\ell-1)2^{n+1}+4h})\\ \ \\
\end{array}
\]
Using \eqref{conjugado diagonal} and the relations
\[
\gamma^{n}_{\ell,h}=\gamma^{n+1}_{\ell,2h},\ \ \ \ \ \frac12(\gamma^n_{\ell,h-1}+\gamma^n_{\ell,h})
=\gamma^{n+1}_{\ell,2h-1},
\]
we have
\begin{eqnarray*}
E_{\cD(2^{k+n})}(A(n+1))&=&E_{\cD(2^{k+n})}(W_{k+n-1}\,(I_2\otimes E_{\cD(2^{k+n-1})}(A(n)))\,W_{k+n-1}^*)
\end{eqnarray*}
\begin{eqnarray*}
&=&\ds\sum_{\ell=1}^{2^{k-1}}\sum_{h=1}^{2^{n-1}}\,
\gamma^{n}_{\ell,h-1}\, f^{k+n}_{(\ell-1)2^{n+1}+4h-3}+ \frac12(\gamma^n_{\ell,h-1}
+\gamma^n_{\ell,h})\,f^{k+n}_{(\ell-1)2^{n+1}+4h-2} \\ \ \\
&&\ds+\frac12(\gamma^n_{\ell,h-1}+\gamma^n_{\ell,h})\,f^{k+n}_{(\ell-1)2^{n+1}+4h-1}
+\gamma^n_{\ell,h}\,f^{k+n}_{(\ell-1)2^{n+1}+4h}\\ \ \\
&=&\ds\sum_{\ell=1}^{2^{k-1}}\sum_{h=1}^{2^{n-1}}\,
\gamma^{n+1}_{\ell,2h-2}\, f^{k+n}_{(\ell-1)2^{n+1}+4h-3}
+ \gamma^{n+1}_{\ell,2h-1}\,f^{k+n}_{(\ell-1)2^{n+1}+4h-2} \\ \ \\
&&\ds+\gamma^{n+1}_{\ell,2h-1}\,f^{k+n}_{(\ell-1)2^{n+1}+4h-1}
+\gamma^{n+1}_{\ell,2h}\,f^{k+n}_{(\ell-1)2^{n+1}+4h}\\ \ \\
&=&\ds\sum_{\ell=1}^{2^{k-1}}\sum_{h=1}^{2^{n}}\,
\gamma^{n+1}_{\ell,h-1}\, f^{k+n}_{2^n(\ell-1)+2h-1}+\gamma^{n+1}_{\ell,h}\,f^{k+n}_{2^n(\ell-1)+2h}\qedhere
\end{eqnarray*}
\end{proof}

\begin{teo}[Carpenter's theorem for some non-discrete operators]\label{iteracion x}
Let $\m$ be a separable II$_1$ factor and let $x\in \a^+$ be the
associated operator to a family $\{p_i^k\}$ of projections in a
semiregular masa $\a$ in $\m$. If $A\in M_{2^k}(\CC)$ then the
sequence $(a_n)_{n\in \NN}\subseteq \m$ given by $a_1=\pi_k(A)$ and
$$a_{n+1}=\pi_{k+n}(A(n+1))=\pi_{k+n}(W_{n+k-1}) \ \pi_{k+n}(A(n))\
\pi_{k+n}(W_{n+k-1})^*$$ converges strongly to an operator
$a\in \m$. Moreover, we have that
\begin{enumerate}
\item if $A$ is projector (resp. self-adjoint, positive) then so is
$a$;
\item if $A_{jj}=d_j$ and $f:[0,1]\rightarrow\CC$ is the piecewise linear function given by
\[
f(t)=d_{2j-1}+2^{k-1}\left(t-\frac{j-1}{2^{k-1}}\right)\,(d_{2j}-d_{2j-1}), \ \ \ \
t\in\left[\frac{j-1}{2^{-(k-1)}},\frac{j}{2^{-(k-1)}}\right),
\]$j=1,\ldots,2^{k-1}$,
then $E_\a(a)=f(x)$;
\item if $B\in M_{2^k}$ and $b=\lim_n \pi_{n+k-1}(B(n))$ then
$\|b-a\|_2^2=\frac{1}{2^k}\|B-A\|_2^2$.
\end{enumerate}
\end{teo}
\begin{proof}
Using Corollary \ref{repre}, Lemma \ref{iteracion} and the fact that
if $C\in M_{2^k}(\CC)$ then $\|\pi_{k+n-1}(C)\|_2^2= 2^{-(k+n-1)}\,
\|C\|_2^2$, we have
$$\|a_{n+1}-a_n\|_2^2 \leq \lambda\,\|a_n-a_{n-1}\|_2^2$$
 with $0<\lambda<1$, independent of $A,\, k$ and $n$. Then the sequence $\{a_n\}$ converges in $\|\cdot\|_2$
  to an operator $a\in\m$.  We now prove the remaining items.

(i) If $A$ is a projector (resp. self-adjoint, positive), then so is $A(n)$, for each $n$.
Since every $\pi_n$ is a $*$-representation,
$\pi_{n+k-1}(A(n))$ inherits the properties from $A$, and any of the three properties passes
to the $\|\cdot\|_2$-limit.

(ii) By Lemmas \ref{repre} and \ref{lema diagonal},
\begin{eqnarray*}
E_\a(a_n)&=&E_\a(\pi_{k+n-1}(A(n)))=\pi_{k+n-1}(E_{\cD(2^{k+n-1})}(A(n)))\\
&=&\ds\sum_{\ell=1}^{2^{k-1}}\sum_{h=1}^{2^{n-1}}\,
\gamma^{n}_{\ell,h-1}\, p^{k+n-1}_{2^n(\ell-1)+2h-1}+\gamma^{n}_{\ell,h}\,p^{k+n-1}_{2^n(\ell-1)+2h}.
\end{eqnarray*}
If we consider the discrete operators $x_n$ as defined in Remark \ref{el x} then
\begin{eqnarray*}
& & x_{k+n-1}=\sum_{i=1}^{2^{k+n-1}}\frac{i}{2^{k+n-1}}\,p^{k+n-1}_i\\
& &\ =\ds\sum_{\ell=1}^{2^{k-1}}\sum_{h=1}^{2^{n-1}}\,
\frac{2^n(\ell-1)+2h-1}{2^{k+n-1}}\, p^{k+n-1}_{2^n(\ell-1)+2h-1}+
\frac{2^n(\ell-1)+2h}{2^{k+n-1}}\,p^{k+n-1}_{2^n(\ell-1)+2h}.
\end{eqnarray*}
It is easy to check that
\[
\frac{\ell-1}{2^{k-1}}\leq\frac{2^n(\ell-1)+2h-1}{2^{k+n-1}}<\frac{2^n(\ell-1)+2h}{2^{k+n-1}}
<\frac{\ell}{2^{k-1}},
\] and, if $\gamma^n_{\ell,h}$ are as in the statement of Lemma \ref{lema diagonal}, then
\begin{eqnarray*}
f\left(\frac{2^n(\ell-1)+2h-1}{2^{k+n-1}}\right)&=&\gamma^n_{\ell,h-1}+\frac1{2^{n}}(d_{2\ell}-d_{2\ell-1}),
\\
f\left(\frac{2^n(\ell-1)+2h}{2^{k+n-1}}\right)&=&\gamma^n_{\ell,h-1}.
\end{eqnarray*}
So
\begin{eqnarray*}
f(x_{k+n-1})&=&\ds\sum_{\ell=1}^{2^{k-1}}\sum_{h=1}^{2^{n-1}}\,
\left(\gamma^n_{\ell,h-1}+\frac1{2^{n}}(d_{2\ell}-d_{2\ell})\right)\, p^{k+n-1}_{2^n(\ell-1)+2h-1}\\ \ \\ &+&
\gamma^n_{\ell,h-1}\,p^{k+n-1}_{2^n(\ell-1)+2h}\\ \ \\
&=&E_\a(a_n)+\sum_{\ell=1}^{2^{k-1}}\sum_{h=1}^{2^{n-1}}\,\frac1{2^{n}}(d_{2\ell}-d_{2\ell-1})
\, p^{k+n-1}_{2^n(\ell-1)+2h-1}.
\end{eqnarray*}
Thus, letting $d=\max\{d_i\}\leq\|A\|$,
\[
\|E_\a(a_n)-f(x_{k+n-1})\|=\|\sum_{\ell=1}^{2^{k-1}}\sum_{h=1}^{2^{n-1}}\,\frac1{2^{n}}(d_{2\ell}-d_{2\ell-1})
\, p^{k+n-1}_{2^n(\ell-1)+2h-1}\|\leq\frac{d}{2^{n}}.
\]
Since $a_n\xrightarrow{\|\cdot\|_2}a$,
$x_n\xrightarrow{\|\cdot\|_2}x$, $E_\a$ is normal, and $f$ is
continuous off a set of Lebesgue measure 0 (see Remark \ref{el x}),
we get $E_\a(a_n)\xrightarrow{\|\cdot\|_2}E_\a(a)$,
$f(x_n)\xrightarrow{\|\cdot\|_2}f(x)$, and so $E_\a(a)=f(x)$.

\medskip

(iii) Note that $\|I_2\otimes A\|_2^2=2\,\|A\|_2^2$. Then we have
\begin{eqnarray*}
& & \|\pi_{n+k-1}(B(n))-\pi_{n+k-1}(A(n))\|_2^2=
 \frac1{2^{n+k-1}}\,\|B(n)-A(n)\|_2^2\\
& & = \frac1{2^{n+k-1}}\,\|W_{k+n-2}(I_2\otimes(B(n-1)-A(n-1))\,W_{k+n-2}\|_2^2\\
& & = \frac1{2^{n+k-2}}\,\|B(n-1)-A(n-1)\|_2^2\\
& & \ \vdots \\
& & = \frac1{2^k}\,\|B-A\|_2^2.
\end{eqnarray*}
By continuity,
\[
\|b-a\|_2^2=\frac1{2^k}\,\|B-A\|_2^2.\qedhere
\]
\end{proof}

The  continuity property in (iii) suggests a possible strategy for
solving Kadison's conjecture in this setting: using the previous
notations, let $g(x)\in \a$ for $g\in \mathrm{L}^\infty([0,1])$,
$0\leq g\leq 1$ and for $k\in \NN$ let $g_k=\sum_{i=1}^{2^k}
g_{i,k}\ \chi_{I_i^k}$ be a sequence \emph{dyadic} discrete
functions, $0\leq g_k\leq 1$, $\int_0^1g_k(t)\, dt=2^{-k} m(k)$ for
some $m(k)\in \NN$ and such that converges to $g$ in
$\mathrm{L}^2([0,1])$. Then, if we were able to construct a sequence
of projection \emph{matrices} $A_k\in M_{2^k}(\CC)$ such that
\begin{equation}\label{en coherente}
\cD_{2^k}(A_k)=\sum_{i=1}^{2^k}g_{i,k} \,f_i^k \ \ \text{ and } \ \ \limsup_{k}
\frac12\,\frac{\|A_{k+1}-I_2\otimes A_k\|^2_2}{\|A_{k}-I_2\otimes A_{k-1}\|^2_2}<1 \end{equation} then, denoting by
$a_k=\lim_n\pi_{k+n}(A_k)$, we would have that
$$a_k\xrightarrow[k]{\|\,\|_2} a, \ \ E_\a(a_k)\xrightarrow[k]{\|\,\|_2} g(x)$$
since by \eqref{en coherente}, $\{a_k\}_k$ would be
a Cauchy sequence of projections in $\|\cdot\|_2$. Hence $a\in\m^+$ would be
a projection such that $E_\a(a)=g(x)$ for an arbitrary  $g\in
\mathrm{L}^\infty([0,1])$, $0\leq g\leq 1$.

\smallskip

\noindent{\bf Acknowledgements.} The second named author wishes to thank D. Farenick and the Department of Mathematics and Statistics
at the University of Regina for warm hospitality received during his stay. We would also like to thank R. Sasyk and S. White for useful
comments.

\end{document}